\documentclass{article}
\usepackage[utf8]{inputenc}

\title{The linearization problem of a binary quadratic problem and its applications}
\author{Hao Hu \thanks{Dept.~of Combinatorics and Optimization,  University of Waterloo, Waterloo, Canada, {\tt h92hu@uwaterloo.ca}}
	\and {Renata Sotirov \thanks{Department of Econometrics and OR, Tilburg University, Tilburg, The Netherlands, {\tt r.sotirov@uvt.nl}}} }

\usepackage[numbers,sort]{natbib}
\bibpunct[, ]{[}{]}{,}{n}{}{,}%

\usepackage{amsfonts}
\usepackage{amssymb}
\usepackage{url,graphicx,tabularx,array,geometry}
\usepackage{amsmath,lipsum}
\usepackage{color}
\usepackage{epstopdf}
\usepackage{verbatim}
\usepackage{blkarray} 
\usepackage{enumitem} 
\usepackage{multirow} 
\usepackage{blkarray}
\usepackage{float}
\usepackage{hyperref}
\usepackage{subfig}
\usepackage{soul} 

\usepackage{algorithm}
\usepackage[noend]{algpseudocode}

\usepackage{geometry}

\usepackage{amsthm} 
\newtheorem{theorem}{Theorem}[section]
\newtheorem{lemma}[theorem]{Lemma}
\newtheorem{proposition}[theorem]{Proposition}
\newtheorem{corollary}[theorem]{Corollary}
\newtheorem{definition}[theorem]{Definition}
\newtheorem{example}[theorem]{Example}
\theoremstyle{remark}
\newtheorem{remark}{Remark}

\newcommand{\R}{\mathbb{R}}

\DeclareMathOperator{\diag}{diag}
\DeclareMathOperator{\Diag}{Diag}

\date{}

\begin{document}

\maketitle

\begin{abstract}
We provide several applications of the linearization problem of a binary quadratic problem.
We propose a new lower bounding strategy, called the linearization-based scheme,
that is based on a simple certificate for a quadratic function to be non-negative on the feasible set.
Each linearization-based bound requires a set of linearizable matrices as an input.
We prove that the Generalized Gilmore-Lawler bounding scheme  for binary quadratic problems provides linearization-based bounds.
Moreover, we show that  the bound obtained from the  first level reformulation linearization technique is also a type of linearization-based bound, which
enables us to provide a comparison among mentioned bounds.
However, the strongest linearization-based bound is the one that uses the full characterization of the set of linearizable matrices.
Finally, we present a polynomial-time algorithm for the linearization problem of the quadratic shortest path problem
on directed acyclic graphs. Our algorithm  gives a complete characterization of the set of linearizable matrices for the quadratic shortest path problem.
\end{abstract}


\section{Introduction}

A binary quadratic problem (BQP) is an  optimization problem with binary variables, quadratic objective function and linear constraints.
BQPs are  NP-hard problems.  A wide range of combinatorial optimization problems,
including the quadratic assignment problem (QAP), the quadratic shortest path problem (QSPP), the graph partitioning problem, the max-cut problem, clustering problems, etc.,
can be modeled as a BQP.

We study here the linearization problem for binary quadratic problems.
A binary quadratic optimization problem is said to be linearizable if there exists a corresponding cost vector such that
the associated costs for both, quadratic and linear problems are equal for every  feasible vector.
The BQP linearization problem asks whether an instance of the BQP is linearizable.
The linearization problem is studied  in the context of many combinatorial optimization problems.
Kabadi and Punnen \cite{kabadi2011n} give a necessary and sufficient condition for an instance of the quadratic assignment problem (QAP) to be linearizable,
 and develop a polynomial-time algorithm to solve the corresponding linearization problem.
The linearization problem for the Koopmans-Beckmann QAP is studied in \cite{punnen2013linear}. Linearizable special cases of the QAP are studied
 in \cite{adams2014linear,cela2016linearizable,punnen2001combinatorial}.
In \cite{custic2017characterization} it is shown  that the linearization problem for the bilinear assignment problem can be solved in polynomial time.
The linearization problem for the quadratic minimum spanning tree problem was considered by \'{C}usti\'{c} and Punnen \cite{custic2018characterization}.
Punnen, Walter and Woods \cite{punnenWoods} provide necessary and sufficient conditions for which a cost matrix of the quadratic traveling salesman problem is linearizable.
Hu and Sotirov \cite{hu2018special} provide a polynomial time algorithm that verifies if a QSPP instance on a directed grid graph  is linearizable.
The authors of \cite{hu2018special} also present necessary conditions for a QSPP instance on complete digraphs to be linearizable. These conditions are also sufficient when the
complete digraph has only four vertices.

There are very few studies concerning applications of the linearization problem.
Punnen, Pandey and Friesen  \cite{punnen2018representations}
show how to derive equivalent representations of a quadratic optimization problem by using linearizable matrices of the problem.
They show that equivalent representations might result with different bounds for the optimization problem.

In this paper, we present several interesting applications of the linearization problem.
We  propose a new lower bounding scheme that uses a simple certificate for a quadratic function to be non-negative on the feasible set.
The resulting bounds we call the linearization-based bounds.
Each linearization-based bound requires a set of linearizable matrices as an input, and its quality depends on those matrices.
To compute a particular linearization-based bound, one needs to solve one linear programming problem.
The strongest linearization-based bound is the one that uses the full characterization of the set of linearizable matrices.
Further,  we show that  bounds obtained from an iterative lower bounding strategy  for BQPs, known as  the  Generalized Gilmore-Lawler  (GGL) scheme,
see e.g.,  \cite{hahn1998lower,carraresi1992new,rostami2018quadratic,Rostami:QMST}, are also  linearization-based bounds.
Note that the well known Gilmore-Lawler bound is the  first bound within the Generalized Gilmore-Lawler bounding scheme.
Furthermore, we prove that one of the linearization-based bounds with a particular choice of linearizable matrices  dominates the GGL bounds.
The same linearization-based bound  is equivalent to the first level RLT relaxation by Adams and Sherali \cite{adams1990linearization}
for BQPs where upper bounds on the vector of variables are implied by the constraints.
 Here RLT stands for reformulation linearization technique.
This result explains the relation  between the Generalized Gilmore-Lawler bounds and the first level RLT bound,
which was already observed in the context of the quadratic assignment problem \cite{frieze1983quadratic,hahn1998branch} but not in general.
Further,  we extend the notion of linearizable matrices, which results in the extended linearization-based bounds.

Finally, we provide a polynomial-time algorithm for the linearization problem of the quadratic shortest path problem on directed acyclic graphs (DAGs).
We solve the linearization problem for the QSPP on DAGs  in ${\mathcal O}(nm^{3})$ time, where $n$ is the number of vertices and  $m$ is the number of arcs in the given graph.
Our algorithm also yields a characterization of the set of linearizable matrices, and thus provides the strongest linearization-based bound for the QSPP on DAGs.

\medskip	
The paper is organized as follows. In Section \ref{section_qspp} and \ref{sec_linearization}, we introduce the binary quadratic problem and its linearization problem, respectively.
 In Section \ref{sec_sos}, we show how to reformulate a binary quadratic minimization problem into an equivalent maximization problem that is suitable for deriving bounds.
In Section \ref{sec:lbb}, we introduce the linearization-based scheme.
In Section \ref{sect:GLB}, we  show that the  Generalized Gilmore-Lawler bounds are also linearization-based bounds.
Section \ref{sect:RLT} relates different linearization-based bounds to the first level RLT bound, and Section \ref{sect:strongest} demonstrates
the strength of the strongest lineariztion-based bound.
In Section \ref{sect:extendLin} introduced the extended linearization-based bounds.
In Section \ref{section_qspp_linear}, we present a polynomial-time algorithm that verifies whether a QSPP instance on a directed acyclic graph is linearizable.
Conclusion and suggestions for further research are given in Section \ref{sec_conclusions}.

\section{Binary quadratic problems} \label{section_qspp}

In this section, we introduce binary quadratic problems and their two special cases; the quadratic assignment problem and the quadratic shortest path problem.

Let $K$ be the set of feasible binary vectors, i.e.,
	\begin{equation} \label{BQP_K}
	K := \{ x \in \mathbb{R}^{m} \;|\; Bx=b, ~x \in \{0,1\}^{m} \},
	\end{equation}
where $B\in \mathbb{R}^{n\times m}$ and $b\in \mathbb{R}^{n}$.
	We are interested in binary quadratic problems of the form
	\begin{equation}\label{poly_quad}
	\min_{x \in K} x^{\mathrm T}Qx,
	\end{equation}
	where $Q \in \mathbb{R}^{m \times m}$ is the given quadratic cost matrix.
Note that we allow here that  $Q$  has  also negative elements.
In the case that $Q$ is a diagonal matrix i.e., $Q={\rm Diag}(c)$ the objective is linear and we have the following linear optimization problem:
	\begin{equation}\label{poly_linear}
	\min_{x \in K} c^{\mathrm T}x.
	\end{equation}
The simple model \eqref{poly_quad} is notable for representing a wide range of combinatorial optimization problems, including the quadratic assignment problem and  the quadratic shortest path problem.

The quadratic assignment problem  is one of the most difficult combinatorial optimization problems. It was introduced by Koopmans and Beckmann \cite{koopmans1957assignment}.
It is well-known that the QAP contains the traveling salesman problem as a special case and is therefore NP-hard in the strong sense.
The QAP can be described as follows.
Suppose that there are $n$ facilities and $n$ locations. The flow between each pair of facilities, say $i,k$, and the distance between each pair of locations, say $j,l$, are specified by $a_{ik}$ and $d_{jl}$, respectively.
 The problem is to assign all facilities to different locations with the goal of minimizing the sum of the distances multiplied by the corresponding flows.
 The quadratic assignment problem  is given by:
	\[
	\min \left \{ \sum_{i,j,k,l}a_{ik} d_{jl}x_{ij}x_{kl} : ~X=(x_{ij}), ~X\in \Pi_n \right \},
	\]
	where $\Pi_n$ is the set of $n\times n $ permutation matrices.
If $A=(a_{ik})$, $D=(d_{jl})$ and $x = \text{vec}(X) \in \mathbb{R}^{n^{2}}$, then the objective can be written as $x^{\mathrm T}(A \otimes D)x$.
Here, the vec operator stacks the columns of the matrix $X$, and $\otimes$ denotes the Kronecker product.
The QAP is  known as a generic model for various real-life problems, see e.g., \cite{Burkard:2009}.

	The quadratic shortest path problem has a lot of important application in transportation problems, see  \cite{murakami1997comparative,rostami2018quadratic,Sen:01,sivakumar1994variance}.
The QSPP is an NP-hard optimization problem, see \cite{hu2018special,rostami2018quadratic}, and it can be described as follows.
Let $G = (V,A)$ be a directed graph with $n$ vertices and $m$ arcs, and $s$, $t$ two distinguished vertices in $G$. A path is a sequence of distinct vertices ${(v_{1},\ldots,v_{k})}$
	such that $(v_{i},v_{i+1})\in A$ for $i=1,\ldots,k-1$. An $s$-$t$ path is a path ${P =(v_{1},v_{2},\ldots,v_{k})}$ from the source vertex $s = v_{1}$ to the target vertex $t = v_{k}$.
Let the interaction cost between two distinct arcs $e$ and $f$ be $2q_{ef}$, and the linear cost of an arc $e$ be $q_{e,e}$. The quadratic shortest path problem is given by:
	\begin{equation}\label{quadratic_QSPP}
	\text{minimize} \big\{ \sum_{e,f \in A} q_{e,f}x_{e}x_{f}  \;|\;  x \in \mathcal{P} \big\},
	\end{equation}
	where $\mathcal{P}$ be the set of characteristic vectors of all $s$-$t$ paths in $G$.

\section{The linearization problem of a BQP}\label{sec_linearization}

We say that the binary quadratic optimization problem \eqref{poly_quad} is {linearizable}, if there exists a cost vector $c$ such that $x^{\mathrm T}Qx = c^{\mathrm T}x$ for every $x \in K$.
If such a cost vector $c$ exists, then we call it a linearization vector of $Q$.
The cost matrix $Q$ is said to be linearizable if its  linearization vector $c$ exists. The linearization problem of a BQP asks whether $Q$ is linearizable, and if yes, provide its linearization vector $c$.

If $Q$ is linearizable, then \eqref{poly_quad} can be equivalently formulated as \eqref{poly_linear}, and the latter could be much easier to solve.
For instance, in the case of the quadratic assignment problem or  the quadratic shortest path problem on directed acyclic graphs
this boils down to solve a linear programming problem.
	
Let us define the  spanning set of linearizable matrices for the given BQP.
\begin{definition}
Let $\{ Q_{1},\ldots,Q_{k} \}$ be a set of matrices such that a cost matrix $Q$ is linearizable if and only if
$Q = \sum_{i=1}^{k}\alpha_{i}Q_{i}$ for some $\alpha \in \mathbb{R}^{k}$.
Then, we say that $\{ Q_{1},\ldots,Q_{k} \}$ span the set of linearizable matrices.
	\end{definition}
In Section \ref{section_qspp_linear}, we show that the spanning set of linearizable matrices for the quadratic shortest path problem on directed acyclic graphs can be generated efficiently.
Thus we have  a complete characterization of the set of linearizable matrices for the QSPP on DAGs.
This is also the case for the bilinear assignment problem, see \cite{custic2017characterization}.
In fact, the authors in \cite{lendl2017combinatorial} show that the set of linearizable cost matrices for combinatorial optimization problems with interaction costs can be characterized by the so-called {constant value property}, under certain conditions.
For the list of  non-trivial binary quadratic problems for which one can find the spanning set of linearizable matrices, see  \cite{lendl2017combinatorial}.
	
In general, it is not clear whether one can find a complete  characterization  of the set of linearizable matrices for a given BQP.
However, it is not difficult to identify a subset of linearizable matrices. For instance, the sum matrix is often found to be linearizable.
We say that a matrix $M \in \R^{m\times n}$ is a sum matrix generated by vectors
$a \in  \R^{m}$ and $b \in  \R^{n}$ if $M_{i,j} = a_{i} + b_{j}$ for every $i = 1,\ldots,m$ and $j= 1,\ldots,n$.
In the quadratic assignment problem, if $A$ or $D$ is a sum matrix, then the corresponding cost matrix is linearizable, see \cite{burkard2012assignment}.
In the quadratic shortest path problem, if every  $s$-$t$ path in the graph  has the same length, then a sum-matrix $Q$ is always linearizable, see \cite{hu2018special,punnen2018representations}.
	
In the latter case, the condition for a matrix to be linearizable depends on the problem structure.
Since we are interested in a lower bounding scheme for general binary quadratic problems, we need a condition for linearizability that is independent of the problem.
The next result provides an universal sufficient condition for a matrix being linearizable, and it is also the key connecting the linearization problem and some of the existing bounds in the literature.

\begin{lemma} \cite{punnen2018representations} \label{poly_gen1}
Consider the  BQP  \eqref{poly_quad}.
For any $Y \in \mathbb{R}^{n \times m}, z \in \mathbb{R}^{m}$, the matrix $Q =  B^{\mathrm T}Y + \Diag(z) \in \mathbb{R}^{m \times m}$ is linearizable
 with linearization vector $c = Y^{\mathrm T}b + z \in \mathbb{R}^{m}$.
\end{lemma}
\begin{proof}
For any $x \in K$, see \eqref{BQP_K}, we have $x^{\mathrm T}Qx =  x^{\mathrm T}(B^{\mathrm T}Y + \Diag(z))x  = (b^{\mathrm T}Y  + z^{\mathrm T})x  = c^{\mathrm T}x$.
\end{proof}
	
The next result shows that adding redundant equality constraints to the system $Bx=b$ does not generate more linearizable matrices in Lemma \ref{poly_gen1}.
\begin{lemma} \label{lemma_BT}
Consider the  BQP  \eqref{poly_quad}.	Assume the equation $a^{\mathrm T}x = d$ is implied by the system $Bx=b$, i.e.,
$a = B^{\mathrm T}\alpha \in \mathbb{R}^{m}$ and $d = b^{\mathrm T}\alpha  \in \mathbb{R}$ for some $\alpha \in \mathbb{R}^{n}$.
For any $Y \in \mathbb{R}^{n \times m}, y \in \mathbb{R}^{m}$, the matrix
\[
Q= \left(\begin{matrix}
	B \\ a^{\mathrm T}
	\end{matrix} \right)^{\mathrm T}
 \left(\begin{matrix}
	Y \\ y^{\mathrm T}
	\end{matrix} \right)
\]
is linearizable with linearization vector $c =  Y^{\mathrm T}b + dy \in \mathbb{R}^{m}$.
	\end{lemma}
	\begin{proof}
From
 $$\left(\begin{matrix}
	B \\ a^{\mathrm T}
	\end{matrix} \right)^{\mathrm T}
 \left(\begin{matrix}
	Y \\ y^{\mathrm T}
	\end{matrix} \right)= B^{\mathrm T}Y + ay^{\mathrm T} =  B^{\mathrm T}Y + B^{\mathrm T}\alpha y^{\mathrm T} = B^{\mathrm T}(Y + \alpha y^{\mathrm T}) \in \mathbb{R}^{m \times m}$$
it follows that  $B^{\mathrm T}(Y + \alpha y^{\mathrm T})$
is linearizable with linearization vector $c =  (Y + \alpha y^{\mathrm T})^{\mathrm T}b$.
\end{proof}
	
	A matrix $M\in \mathbb{R}^{m \times m}$ is called a weak sum matrix if we can find vectors $a, b\in \mathbb{R}^{m}$ such that $M_{i,j} = a_{i} + b_{j}$ for every $i\neq j$ and $i,j = 1,\ldots,m$.
In particular, if $M$ is a symmetric weak sum matrix, then we can assume $a=b$ in the definition.
	If $e^{\mathrm T}x = \sum_{i}x_{i}$ is a constant for every $x \in K$, then every weak sum matrix in the corresponding BQP is linearizable, see \cite{punnen2018representations}.
 In such case, it follows  from Lemma \ref{lemma_BT} that symmetric weak sum matrices can be represented as symmetrized matrices from  Lemma \ref{poly_gen1}.
 In particular, we have the following corollary.
	
	\begin{corollary}
Consider the  BQP  \eqref{poly_quad}. Assume $Bx=b$ implies that $e^{\mathrm T}x$ is a constant.
 Then every symmetric weak sum matrix $M$ can be written as $M = B^{\mathrm T}Y + Y^{\mathrm T}B  + \Diag(z)$ for some $Y$ and $z$.
 \end{corollary}
	\begin{proof}
	 By assumption, we have $e = B^{\mathrm T}\alpha \in \mathbb{R}^{m}$ for some $\alpha \in \mathbb{R}^{n}$. Since $M$ is a symmetric weak sum matrix, we can write $M = ae^{\mathrm T} + ea^{\mathrm T} + \Diag(z)\in \mathbb{R}^{m \times m}$ for some vectors $a\in \mathbb{R}^{m},z\in \mathbb{R}^{m}$ and all-ones vector $e\in \mathbb{R}^{m}$. Let $Y = \alpha a^{\mathrm T} \in \mathbb{R}^{n \times m}$. Then $$B^{\mathrm T}Y + Y^{\mathrm T}B  + \Diag(z) = B^{\mathrm T}\alpha a^{\mathrm T} + a\alpha^{\mathrm T}B  + \Diag(z) =e a^{\mathrm T} + ae^{\mathrm T}  + \Diag(z) = M. $$
	\end{proof}
	
 Erdo{\u{g}}an and Tansel \cite{erdougan2011two} identify an additively decomposable class of costs for the quadratic assignment problem that provide  a set of linearizable matrices for the QAP. We verified numerically that the linearizable matrices from  \cite{erdougan2011two}  can be written as $B^{\mathrm T}Y + Y^{\mathrm T}B  + \Diag(z)$ for some small instances with $n \leq 10$.

\section{General bounding approaches   }\label{sec_sos}

We  present here an equivalent reformulation of the BQP \eqref{poly_quad} and list several possible bounding approaches for this reformulation.
The new equivalent reformulation of the BQP is  also the basis for deriving our bounding scheme  in the next section.

Let $Q_{1},\ldots,Q_{k}$ be linearizable matrices for a given BQP with linearization vectors $c_{1},\ldots,c_{k}$, respectively.
For example,  linearizable matrices can be obtained from Lemma \ref{poly_gen1} for any binary quadratic problems, or from Proposition \ref{main_basis} for the quadratic shortest path problem.
Define the linear operator $\mathcal{A}: \mathbb{R}^{k} \rightarrow \mathbb{S}^{m}$ by $\mathcal{A}(\alpha) := \sum_{i=1}^{k}\alpha_{i}Q_{i}$ and $C := \begin{bmatrix}
c_{1},\ldots,c_{k}
\end{bmatrix} \in \mathbb{R}^{m \times k}$. Clearly,  $C\alpha$ is a linearization vector of the linearizable cost matrix $\mathcal{A}(\alpha)$, for any $\alpha \in \mathbb{R}^{k}$.

Let $f(x) = x^{\mathrm T}Qx$ and $h_{\alpha,\beta}(x) := x^{\mathrm T}\mathcal{A}(\alpha)x + \beta$, where $\alpha \in \mathbb{R}^{k}$, $\beta \in \mathbb{R}$. Let us reformulate the binary quadratic optimization problem \eqref{poly_quad} equivalently as
\begin{equation}\label{poly_refor2}
\begin{array}{ll}
\sup_{\alpha,\beta} & \min_{x \in K} h_{\alpha,\beta}(x) \\[1.ex]
{\rm s.t.}  &  f -h_{\alpha,\beta} > 0, \;\; \forall x \in K .
\end{array}
\end{equation}
\begin{theorem}\label{BQP_eq}
	The binary quadratic program \eqref{poly_quad} is equivalent to \eqref{poly_refor2}.
\end{theorem}
\begin{proof}
	Let $x^{*}$ be an optimal solution of \eqref{poly_quad} with optimal value $f^{*}$. If $\alpha = 0$ and $\beta = f^{*} - \epsilon$ for $\epsilon > 0$, then they are feasible for \eqref{poly_refor2} as $f -h_{\alpha,\beta} > 0$ for  $x \in K$. As  $f^{*} - \epsilon = \min_{x \in K} h_{\alpha,\beta}(x)$, taking $\epsilon \rightarrow 0$, we conclude that the optimal value of \eqref{poly_refor2} is at least $f^{*}$. Conversely, $f(x) -h_{\alpha,\beta}(x) > 0$ on $K$ implies that $\min_{x \in K} h_{\alpha,\beta}(x)$ is at most $f^{*}$. This shows the equivalence between \eqref{poly_quad} and \eqref{poly_refor2}.
\end{proof}

 We note that Theorem \ref{BQP_eq} holds for any choice of the linearizable matrices in the definition of $\mathcal{A}$.
 Let
 \begin{equation} \label{Kbar}
 \bar{K}=\{ x\geq 0: {B}x={b}\}.
 \end{equation}
From now on we assume w.l.g.~that  ${B}x={b}$ includes also one as an upper bound on $x_i$ for all $i$.
For the inner minimization problem of \eqref{poly_refor2} we have:
\begin{align}
	\underset{x \in \mathbb{R}^{m}}{\min}\{ h_{\alpha,\beta}(x)  \;|\; x \in K \} & = \underset{x \in \mathbb{R}^{m}}{\min}  \{ (C\alpha)^{\mathrm T}x + \beta \;|\; x \in K \} \nonumber\\
& \geq \underset{x \in \mathbb{R}^{m}}{\min} \{ (C\alpha)^{\mathrm T}x + \beta \;|\; x\in \bar{K} \} \label{inn_ineq}\\
& = \underset{y\in \mathbb{R}^{n}}{\max}  \{ b^{\mathrm T}y + \beta \;|\;  B^{\mathrm T}y \leq C\alpha  \}. \nonumber
\end{align}
Here, the first equality exploits the linearizability of $\mathcal{A}(\alpha)$, and the last equality follows from strong duality of linear programming.
In the case that $K$ is the convex hull of the feasible integer points,  like in the case of the linear assignment problem and the shortest path problem  on directed acyclic graphs,
  the above inequality turns to be equality.   However the inequality in \eqref{inn_ineq} can be strict in general.
  Nevertheless, the above  leads to the following optimization problem:
\begin{equation}\label{poly_refor3}
\begin{array}{ll}
\sup_{\alpha,\beta,y} &  b^{\mathrm T}y + \beta \\[1.5ex]
{\rm s.t.}  &  f -h_{\alpha,\beta} > 0, \;\; \forall x \in K \\[1.5ex]
& B^{\mathrm T}y \leq C\alpha,
\end{array}
\end{equation}
that can be exploited to obtain bounds for the BQP \eqref{poly_quad}.
Moreover, several approaches may be used  to compute bounds  for the BQP  \eqref{poly_quad} by exploiting \eqref{poly_refor3}.
For instance, one can  relax the condition $f - h_{\alpha,\beta}\geq 0$ on $K$ by a sum of squares (SOS) decomposition of $f -h_{\alpha,\beta}$.
  In e.g., \cite{lasserre2001global,nie2007complexity,laurent2009sums},
  the authors do not consider the linearization problem and construct hierarchies of approximations based on sum of squares decompositions
  for the following optimization problem:
  \[
  \sup \{ \beta \;|\; f(x)-\beta \geq 0, \; \forall x \in K \},
  \]
where $f(x)$ is a multivariate polynomial and $K$ the closed semialgebraic set.

Buchheim et al., \cite{buchheim2015quadratic} propose a semidefinite programming relaxation for binary quadratic programs.
Their approach can be viewed as a special case of \eqref{poly_refor2}, where $\mathcal{A}(\alpha)$ is a diagonal matrix, and
$Q-\mathcal{A}(\alpha) \succeq 0$ is the positivity certificate.
Alternatively, one can use a  simple positivity certificate to derive linear programming lower bounds like we do in the next section.

\section{The linearization-based scheme}\label{sec:lbb}

In this section, we consider a simple but efficient positivity certificate using the fact that both $f$ and $h_{\alpha,\beta}$ are quadratic functions.
This yields an efficient lower bounding scheme. Here, we use the same notation as in the  previous section, namely
$f(x) = x^{\mathrm T}Qx$ and $h_{\alpha,\beta}(x) = x^{\mathrm T}\mathcal{A}(\alpha)x + \beta$.

Note that if $\mathcal{A}(\alpha)\leq Q$ and $\beta\leq 0$, then  $f-h_{\alpha,\beta} \geq 0$ for all $x\in K$. This leads to the following relaxation for the BQP \eqref{poly_quad}:
\begin{equation}\label{poly_lbb}
\begin{array}{ll}
& \underset{\alpha,y}{\max} \{ b^{\mathrm T}y \;|\;  B^{\mathrm T}y \leq C\alpha, \; \mathcal{A}(\alpha)\leq Q\},
\end{array}
\end{equation}
where we have removed the redundant scalar variable $\beta$. We call relaxation \eqref{poly_lbb} the linearization-based relaxation.
Thus, a linearization-based bound (LBB) is  a solution of one linear programming problem.
 However, the quality of so obtained bound depends on the choice of $\mathcal{A}(\alpha)$.
For example, one can take  linearizable matrices from Lemma \ref{poly_gen1} for any BQP.
 De Meijer and Sotirov \cite{deMeijer2019} choose a particular linearizable matrix for the quadratic cycle problem, which enables them to compute strong bounds fast.

Billionnet et al.~\cite{billionnet2009improving} provide a method to  reformulate  a binary quadratic program
into an equivalent binary quadratic program  with a convex quadratic objective function.
Their approach results with the tightest  convex bound.
However the authors from \cite{billionnet2009improving} need to solve a semidefinite programming relaxation in order to compute their bound,
which makes their approach computationally more expensive than the approach we present here.

Next, we consider reformulations of the BQP and their influence to  the  linearization-based bounds.
For any skew-symmetric matrix $S$ and vector $d$, we can reformulate the objective function of \eqref{poly_quad}
as $f(x) =  x^{\mathrm T}(Q+S+\Diag(d))x - d^{\mathrm T}x$ using the fact that $x\in K$ is binary.
This follows from the fact that $x^{\mathrm T}Sx = 0$ and $x_{i}^{2}=x_{i}$, see also Theorem 2.4 in \cite{punnen2018representations}.
Thus, we obtain  an equivalent representation of  \eqref{poly_quad}.
Since we have an extra linear term $d^{\mathrm T}x$, we  underestimate
 $f(x)$ by $h_{\alpha,\gamma}(x) := x^{\mathrm T}\mathcal{A}(\alpha)x + \gamma^{\mathrm T}x$. So both $f(x)$ and $h_{\alpha,\gamma}(x)$ have an extra linear term now.
 Then $x^{\mathrm T}\mathcal{A}(\alpha)x + \gamma^{\mathrm T}x  \leq f(x)$ if $\mathcal{A}(\alpha) \leq Q+S+\Diag(d)$ and $\gamma \leq - d$.
  Under this setting, we can derive a bound in the same way as in \eqref{poly_lbb}. In particular, we have
	\begin{align}
	& \max_{\alpha,\gamma} \left\{  \min_{x \in K} (C\alpha + \gamma)^{\mathrm T}x  \;|\; \mathcal{A}(\alpha) \leq Q+S+\Diag(d),\;  \gamma \leq - d \right \} \nonumber \\
	& =   \max_{\alpha,\gamma,y} \left \{ b^{\mathrm T}y  \;|\; B^{\mathrm T}y \leq C\alpha + \gamma, \; \mathcal{A}(\alpha) \leq Q+S+\Diag(d), \; \gamma \leq - d \right \} \label{invariant_for}.
	\end{align}
Below we show that the bound \eqref{invariant_for},  is invariant under a reformulation of the objective function
 when  $\mathcal{A}(\alpha)$ is in the special form, i.e., contains  every skew-symmetric matrix.
Note that skew-symmetric matrices are linearizable.

\begin{proposition}\label{invariant_lemma}	
Assume that for any skew-symmetric matrix $S \in \mathbb{R}^{m \times m}$ and vector $d \in \mathbb{R}^{m}$
there exists  $\alpha \in \mathbb{R}^{k}$ such that $\mathcal{A}(\alpha) = S + \Diag(d)$ in \eqref{invariant_for}. For the binary quadratic optimization problem
	\begin{equation} \label{invariant_obj}
	\min_{x \in K} x^{\mathrm T}(Q+S+\Diag(d))x - d^{\mathrm T}x,
	\end{equation}
the bound \eqref{invariant_for} does not depend on the choice of the skew-symmetric matrix $S$ or the vector $d$.
\end{proposition}
\begin{proof}
Let $S:=S_1$ and $d:=d_1$ in the BQP \eqref{invariant_obj}, and
$(\alpha^{*},\gamma^{*},y^{*})$ be a feasible solution of \eqref{invariant_for}, whose objective value $f^{*}$ is given by
	\begin{equation}\label{invar1}
	\begin{array}{ll}
	f^{*} = \underset{y}{\max} \{ b^{\mathrm T}y \;|\; B^{\mathrm T}y \leq C\alpha^{*} + \gamma^{*} \}.\\
	\end{array}
	\end{equation}
	Let $S:=S_{2}$ and $d:=d_{2}$ in the BQP \eqref{invariant_obj}. We now construct a feasible solution of \eqref{invariant_for} for \eqref{invariant_obj} having the objective value $f^{*}$.
 By assumption, we can find $\hat{\alpha} \in \mathbb{R}^{k}$ such that $\mathcal{A}(\hat{\alpha}) = S_{2}-S_{1} + \Diag(d_{2}-d_{1})$.
 Let $\alpha^{**} = \alpha^{*}+\hat{\alpha}$ and $\gamma^{**} = \gamma^{*}-d_{2}+d_{1}$.  Note that $C\alpha^{**} = C\alpha^{*} + d_{2} - d_{1}$,
  and thus $C\alpha^{**} +\gamma^{**} = C\alpha^{*} +\gamma^{*}$.
  Here we assume w.l.g.~that  the linearization vector of a diagonal matrix $\Diag(d)$ is given by $d$.
   To see that $(\alpha^{**},\gamma^{**},y^{*})$ is a feasible solution of \eqref{invariant_for}, we check
	$$\begin{array}{cll}
	B^{\mathrm T}y^{*} &\leq C\alpha^{*} + \gamma^{*} &= C\alpha^{**} + \gamma^{**},\\
	\mathcal{A}(\alpha^{**})&=\mathcal{A}(\alpha^{*}) + \mathcal{A}(\hat{\alpha}) & \leq Q+S_{2}+\Diag(d_{2}), \;\\
	\gamma^{**} &= \gamma^{*}-d_{2}+d_{1} & \leq - d_{2}.
	\end{array}$$
	Furthermore, the objective value of this solution is clearly $f^{*} = b^{\mathrm T}y^{*}$.
\end{proof}

We remark here that in the case that  $\mathcal{A}(\alpha)$  does not contain skew-symmetric matrices, the linearization-based bound \eqref{invariant_for} depends on $S$.
Note that in Proposition \ref{invariant_lemma} we do not specify a non-skew-symmetric part of $\mathcal{A}({\alpha})$.

Let us now restrict to  $B^{\mathrm T}Y +\Diag(z)$ where $Y \in \mathbb{R}^{n \times m}$ and $z \in \mathbb{R}^{m}$, see  Lemma \ref{poly_gen1}.
Now, for any skew-symmetric $S \in \mathbb{R}^{m \times m}$ we have that   $B^{\mathrm T}Y + S+  \Diag(z)$ is linearizable with linearization vector $Y^{\mathrm T}b  + z$.
 This yields the following linearization-based relaxation:
\begin{equation}\label{invariant_lbb1}
\underset{Y,S,z,y}{\max} \{ b^{\mathrm T}y \;|\;     B^{\mathrm T}y \leq Y^{\mathrm T}b  + z, \; B^{\mathrm T}Y + S+  \Diag(z) \leq Q, \; S+S^{\mathrm T} = 0 \}.
\end{equation}
This relaxation satisfies the assumption of Proposition \ref{invariant_lemma}. Thus, we may assume without loss of generality that $Q$ is symmetric for \eqref{invariant_lbb1}.
By applying Proposition \ref{invariant_lemma}, the next result shows that the variable $S$ can be eliminated from \eqref{invariant_lbb1} to obtain an equivalent relaxation with less variables and constraints.

\begin{proposition}
Assume $Q$ is symmetric. The relaxation \eqref{invariant_lbb1} is equivalent to
\begin{equation}\label{poly_lbb1}
\underset{Y,z,y}{\max} \{ b^{\mathrm T}y \;|\;    B^{\mathrm T}y \leq 2Y^{\mathrm T}b  + z, \; B^{\mathrm T}Y + Y^{\mathrm T}B  + \Diag(z) \leq Q \}.
\end{equation}
\end{proposition}
\begin{proof}
Let $(Y,S,z,y)$ be a feasible solution of \eqref{invariant_lbb1}. Then $\frac{1}{2}(B^{\mathrm T}Y + Y^{\mathrm T}B) +  \Diag(z) \leq Q$ as $Q$ is symmetric, and thus $(\frac{1}{2}Y,z,y)$ is feasible for \eqref{poly_lbb1}. Conversely, let $(Y,z,y)$ be feasible for \eqref{poly_lbb1}.
Let $S= Y^{\mathrm T}B -B^{\mathrm T}Y$. Then $S$ is skew-symmetric, and
\[
2B^{\mathrm T}Y + S  + \Diag(z) =  B^{\mathrm T}Y + Y^{\mathrm T}B  + \Diag(z)\leq Q.
\]
Therefore $(2Y,S,z,y)$ is a feasible solution for \eqref{invariant_lbb1}.
\end{proof}

In what follows, we consider the following two cases. In both cases, we can assume $Q$ is symmetric without loss of generality.
\begin{itemize}
\item
We replace
$\mathcal{A}(\alpha)$ by $B^{\mathrm T}Y +  Y^{\mathrm T}B+  \Diag(z)$, and $C\alpha$ by $2Y^{\mathrm T}b  + z$
 in  \eqref{poly_lbb}.  Thus, we have the linearization-based relaxation \eqref{poly_lbb1}.
 The obtained lower bound, denoted by $v_{LBB'}$, is the optimal solution of  \eqref{poly_lbb1}.

 \item If we know the spanning set of linearizable matrices $\{Q_{1},\ldots,Q_{k}\}$, then we denote by $v_{LBB^{*}}$ the corresponding  linearization-based bound.
\end{itemize}

\section{The LBB and related bounds} \label{sec:hs}

In Section \ref{sect:GLB} we present the Generalized Gilmore-Lawler bounding scheme, that is a well known iterative lower bounding scheme for binary quadratic problems.
 We  show that the GGL bounds are dominated by  our linearization-based bound $v_{LBB'}$.
 In Section \ref{sect:RLT}, we compare our linearization-based bounds with the bounds obtained from  the first level RLT relaxation proposed by Adams and Sherali \cite{adams1986tight, adams1990linearization}.
In Section \ref{sect:strongest} we show the strength of our strongest linearization-based bound, and in Section \ref{sect:extendLin} we introduce extended-linearization
based bounds.

\subsection{The Generalized Gilmore-Lawler bounding scheme } \label{sect:GLB}

The  Generalized Gilmore-Lawler bounding scheme   is implemented for many optimization problems, including the quadratic assignment problem  \cite{hahn1998lower,carraresi1992new},
the quadratic shortest path problem \cite{rostami2018quadratic}, the quadratic minimum spanning tree problem \cite{Rostami:QMST}.

	Let $Q = [ {q}_{1},\ldots, {q}_{m}] \in \mathbb{R}^{m \times m}$ be the given quadratic cost matrix, see \eqref{poly_quad}.
 Denote by $I_{k}$ the $k$-th column of the identity matrix of size $m$.
  Let $\bar{y}_k \in \R^{n}$ and $\bar{z}_{k} \in \R$ be an optimal solution of the following linear program:
	\begin{align}
	\underset{{y}_{k} \in \R^{n}, z_{k} \in \R}{\max} &\{ \;  b^{\mathrm T} {y}_{k} + z_{k} \;|\;  B^{\mathrm T} {y}_{k} + I_{k}z_{k}\leq {q}_{k} \; \}, \label{GLL_dual}
	\end{align}
	for each $k = 1,\ldots,m$. Collect all $\bar{y}_{k}$ in matrix  $\bar{Y} \in \mathbb{R}^{n \times m}$, and  all $\bar{z}_{k}$  in vector  $\bar{z} \in \mathbb{R}^{m}$.
Define
	\begin{equation}\label{gl_qc}
	\begin{array}{cl}
	\bar{c} &=  \bar{Y}^{\mathrm T}b  + \bar{z} \in \mathbb{R}^{m}, \\
	\bar{Q} &=  B^{\mathrm T}\bar{Y} + \Diag(\bar{z}) \in \mathbb{R}^{m \times m}.
	\end{array}
	\end{equation}
From Lemma \ref{poly_gen1}, we know that  $\bar{c}$ is a linearization vector of  $\bar{Q}$.
The feasibility of \eqref{GLL_dual} implies that $\bar{Q} \leq Q$, and thus
$\min_{x \in \bar{K}} \bar{c}^{\mathrm T}x$, where   $\bar{K}$ is given in \eqref{Kbar}
is a lower bound for the binary quadratic program \eqref{poly_quad}.
Moreover, this bound is known as the Gilmore-Lawler (GL) type bound and it is implemented for many BQP problems including the QAP and the QSPP.
The GL bound was originally  introduced  for the QAP,  see \cite{gilmore1962optimal,lawler1963quadratic}.

Let us call the dual-update of $Q$ the following update of the objective $Q\leftarrow Q - \bar{Q}$.
The dual-update can be applied iteratively followed by some equivalent representation of the objective in order  to obtain an increasing sequence of lower bounds.
This iterative bounding scheme  is known as the Generalized Gilmore-Lawler  bounding scheme, and it is an important  lower bounding strategy for binary quadratic problems.
We describe this bounding scheme below.

	\begin{algorithm}[H]
		\caption{The Generalized Gilmore-Lawler bound}\label{GGL}
		\begin{algorithmic}[1]
			\State \textbf{Input:} The binary quadratic program \eqref{poly_quad}.
			\State \textbf{Output:} The GGL bound $v_{GGL}$.
			\State $i \leftarrow 0$, $Q_{0} \leftarrow Q$,  $c \leftarrow 0$
			\While{true}
			\State 	Set $Q_{i+1} \leftarrow Q_{i} - \bar{Q}_{i}$ and  $c \leftarrow c + \bar{c}_{i}$ from \eqref{gl_qc}\;
			\State 	Set $Q_{i+1} \leftarrow Q_{i+1} + S_{i+1},$ where $S_{i+1}$ is skew-symmetric\;
			
			\If{$||\bar{c}_{i}||=0$}
				\State break
			\EndIf
			\State \textbf{end if}
 \State 	$i$ $\leftarrow$ $i + 1$
			\EndWhile

			\State \textbf{end while}
			\State 	$v_{GGL} \leftarrow \min_{x \in \bar{K} } c^{\mathrm T}x$\;
		\end{algorithmic}
	\end{algorithm}

	Note that the skew-symmetric matrix $S_{i+1}$ in the algorithm yields an equivalent representation. For example, Frieze and Yadegar \cite{frieze1983quadratic}
pick a skew-symmetric matrix $S$ such that $Q+S$ is upper triangular, while Rostami et al., \cite{rostami2018quadratic} keep $Q+S$ symmetric. 
More sophisticated reformulation can also be found by exploiting the problem structure, see Hahn and Grant \cite{hahn1998lower} and de Meijer and Sotirov \cite{deMeijer2019}.
Bounds based on the dual-update are very competitive, see also \cite{carraresi1992new,assad1985lower,hahn1998branch}.
However, the quality of bounds  depends on the choice of the skew-symmetric matrix.

The  key observation here is that each iteration of the Generalized Gilmore-Lawler bounding strategy
is based on maximizing $\bar{c}_{k}$ over $y_{k}$ and $z_{k}$ in a `local' way, i.e., solving $m$ linear programs \eqref{GLL_dual} independently.
Then, followed by the dual update, the quadratic cost matrix is reshuffled by some skew-symmetric matrix, and the procedure is repeated.
 Instead, the linearization-based bound $v_{LBB'}$, see \eqref{poly_lbb1},
   is obtained by maximizing $\min_{x \in \bar{K} } \bar{c}^{\mathrm T}x$ over $Y$, $S$ and $z$ in a ``global" way.
   This means, $v_{LBB'}$ is  optimal in terms of the Generalized Gilmore-Lawler bound.
   The following proof, where we show that $v_{LBB'}$ is stronger than the Generalized Gilmore-Lawler bound $v_{GGL}$,  makes this  observation precise.

\begin{theorem}\label{GLL_lbb}
		$v_{GGL} \leq v_{LBB'}$.
\end{theorem}

\begin{proof}	
Let $\bar{Q}_{i} = B^{\mathrm T}\bar{Y}_{i} + \Diag(\bar{z}_{i})$ $(i=0,\ldots,k-1)$ be the sequence of linearizable matrices from Algorithm \ref{GGL}.
Let $S_{0}$ be  the matrix of zeros.
For convenience, denote by $Q_{1},\ldots ,Q_{k-1}$ the cost matrix obtained right after line six in Algorithm 1.
Thus, $Q_{i} = Q_{i-1} +S_{i-1} - \bar{Q}_{i-1}$ in this notation.
  Define function $g(M) = (M+M^{\mathrm T})/2$. Then $g(M)$ is a zero matrix whenever $M$ is skew-symmetric.
		
From the feasibility of \eqref{GLL_dual}, it holds that $Q_{k} = Q_{k-1}+S_{k-1} - \bar{Q}_{k-1} \geq 0$ and thus
$g(Q_{k}) \geq 0$. As $Q_{k} = Q_{k-1} + S_{k-1} - \bar{Q}_{k-1}  = Q_{0} +  \sum_{i=0}^{k-1} (S_{i} - \bar{Q}_{i})$, we have
		\begin{equation*}
		\begin{array}{l}
		0 \leq g(Q_{k}) = g \big( Q_{0} +  \sum\limits_{i=0}^{k-1} (S_{i} - \bar{Q}_{i}) \big ) = g(Q_{0})- g( \sum\limits_{i=0}^{k-1}\bar{Q}_{i}),
		\end{array}
		\end{equation*}
from where it follows $g\big (\sum_{i=0}^{k-1}\bar{Q}_{i} \big ) \leq g(Q_{0})$.

Let $\tilde{Y} := \frac{1}{2}\sum_{i=0}^{k-1}\bar{Y}_{i}$ and $\tilde{z} := \sum_{i=0}^{k-1}\bar{z}_{i}$. It is easy to see that
\begin{equation*}
\begin{array}{ccc}
g \big (\sum\limits_{i=0}^{k-1}\bar{Q}_{i} \big ) = g \big (B^{\mathrm T}\sum\limits_{i=0}^{k-1}\bar{Y}_{i}+\Diag(\sum\limits_{i=0}^{k-1}\bar{z}_{i}) \big)
= B^{\mathrm T}\tilde{Y}+\tilde{Y}^{\mathrm T}B+\Diag(\tilde{z}).
\end{array}
\end{equation*}
Therefore $B^{\mathrm T}\tilde{Y}+\tilde{Y}^{\mathrm T}B+\Diag(\tilde{z}) \leq g(Q)$.
Note that $Q = g(Q)$ in $v_{LBB'}$.
 Since $c = \sum_{i=0}^{k-1}\bar{c}_{i} = 2\tilde{Y}^{\mathrm T}b  + \tilde{z}$, it holds that
\begin{equation*}
v_{LBB'} \geq \underset{y}{\max} \{ b^{\mathrm T}y \;|\;   B^{\mathrm T}y \leq 2\tilde{Y}^{\mathrm T}b  + \tilde{z} \} = \min_{x \in \bar{K} } c^{\mathrm T}x = v_{GGL}.
\end{equation*}
	\end{proof}
In the next section we relate $v_{LBB'}$ and the bound obtained from the first level RLT.

\subsection{The first level RLT bound} \label{sect:RLT}

The reformulation linearization technique proposed by Adams and Sherali \cite{adams1986tight, adams1990linearization}
generates a hierarchy of linear programming relaxations for binary quadratic programs.
It has been substantiated that this hierarchy generates tight relaxations even at the first level in many applications.

We show here that the linearization-based bound $v_{LBB'}$ coincides with the first level RLT bound
for optimization problems where the constraint $x \leq e$ is redundant for $\bar{K}=\{ x\geq 0:~Bx=b \}$,
see Lemma \ref{LBBisRLT}. If this is not the case, we establish a relation between those two bounds in Lemma \ref{notRedundant}.

The first level RLT relaxation for the binary quadratic  problem \eqref{poly_quad} is given
  as follows:
\begin{equation}\label{RLT}
\begin{array}{rrl}
v_{RLT_1}:=& \underset{x \in \mathbb{R}^{m}, X \in \mathcal{S}^{m} }{\min} &  \langle Q,X \rangle \\[1ex]
&&  (x,X)  \in \mathcal{F} \\[1ex]
&& e - x \geq 0 \\[1ex]
&& J - xe^{\mathrm  T} - ex^{\mathrm T} + X \geq 0 \\[1ex]
&& xe^{\mathrm T} - X \geq 0,
\end{array}
\end{equation}
where  $J$ is all-ones matrix, $\mathcal{S}^{m}$  denotes the set of symmetric matrices of order $m$, and
\begin{equation}\label{RLT_fea}
\mathcal{F}: = \left \{ (x,X) \in (\mathbb{R}^{m},\mathcal{S}^{m}) \;|\; Bx=b  ,\; BX = bx^{\mathrm T}, \; x = \diag(X), \;  x \geq 0, \;  X\geq 0 \right \}.
\end{equation}
Here the `diag' operator maps an $m\times m $ matrix to the $m$-vector given by its diagonal.

 Depending on the specific problem structure the constraint $x\leq e$ can be omitted without affecting its continuous relaxation.
For instance, this is the case when the BQP under consideration is the QAP, the QSPP on directed acyclic graphs, the quadratic cycle cover problem, 
the quadratic minimum spanning tree problem, etc.
 In such cases,  the first level RLT relaxation is given by
\begin{equation}\label{poly_RLT}
\begin{array}{rcl}
v_{RLT_1'} :=& \underset{x \in \mathbb{R}^{m}, X \in \mathcal{S}^{m} }{\min} & \langle Q,X \rangle \\
&\text{ s.t. }& (x,X)  \in \mathcal{F}.
\end{array}
\end{equation}
In general, the bound $v_{RLT_1'}$ is weaker than $v_{RLT_1}$. The next result shows that the linearization-based bound $v_{LBB'}$ is equivalent to $RLT_1'$.

\begin{lemma}\label{LBBisRLT}
Let  $x \leq e$ be redundant for $\bar{K}$, then $v_{LBB'} =v_{RLT_1'}=v_{RLT_1}$.
\end{lemma}
\begin{proof}
The proof follows directly from the dual of \eqref{poly_lbb1}. The Lagrangian function of \eqref{poly_lbb1} is given by
	\begin{equation*}
	\begin{array}{ll}
	{\mathcal L}(Y,z,y,x,X) &= b^{\mathrm T}y + \langle Q - B^{\mathrm T}Y -  Y^{\mathrm T}B -  \Diag(z), X \rangle + \langle 2Y^{\mathrm T}b  + z -  B^{\mathrm T}y, x \rangle \\
	&= \langle Q,X \rangle + \langle b - Bx, y \rangle + \langle 2bx^{\mathrm T} -BX - BX^{\mathrm T}, Y \rangle + \langle x - \diag(X),z \rangle,
	\end{array}
	\end{equation*}
	where $x \in \mathbb{R}^{m}$ and $X \in \mathcal{S}^{m}$. Thus the dual of \eqref{poly_lbb1} is
		\begin{equation}\label{poly_lbb1dual}
	\begin{array}{rl}
	& \underset{x \in \mathbb{R}^{m}, X \in \mathcal{S}^{m} }{\min} \{ \langle Q,X \rangle \;|\; Bx=b ,\;  x \geq 0 ,\; BX = bx^{\mathrm T}, X\geq 0, \; x = \diag(X) \}.\\
	\end{array}
	\end{equation}
As we assumed $Q$ is symmetric for the linearization-based bound  $v_{LBB'}$,
and thus the dual program of  \eqref{poly_lbb1} is exactly the same as the RLT relaxation \eqref{poly_RLT}.
The equality $v_{RLT_1'}=v_{RLT_1}$ follows from e.g., \cite{adams1986tight}.
\end{proof}

\begin{remark}
Note that  we can strengthen the linearization-based bound   $v_{LBB'}$   by exploiting the sparsity pattern of the binary quadratic program.
Let $\mathcal{G} = \{ (i,j) \;|\; x_{i}x_{j} =0 \; \forall x \in K  \}$. Then, $\mathcal{A}(\alpha) \leq Q$ can be equivalently replaced by $\big(\mathcal{A}(\alpha)\big)_{ij} \leq Q_{ij}$ for every $(i,j) \notin \mathcal{G}$. For Lemma \ref{LBBisRLT}, this implies that the dual variable $X$ in \eqref{poly_lbb1dual} has to satisfy $X_{ij} = 0$ for every $(i,j) \in \mathcal{G}$.
Thus, so strengthened linearization-based bound is
equivalent to the first level RLT relaxation \eqref{poly_RLT} with extra sparsity constraints $X_{ij} = 0$ for every $(i,j) \in \mathcal{G}$.
\end{remark}

Theorem \ref{GLL_lbb} proves that the Generalized Gilmore-Lawler bound $v_{GGL}$ is bounded by the first level RLT relaxation for
 a BQP when the upper bound on $x$ is redundant.
A relation between these two bounds was studied only in the context of the  quadratic assignment problem.
In particular, Frieze and Yadegar \cite{frieze1983quadratic} show that the Gilmore-Lawler bound   with a particular decomposition is weaker than
the Lagrangian relaxation of their well known mixed-integer linear programming formulation for the QAP.
On the other hand, the linear programming relaxation of the mentioned mixed-integer linear programming  formulation is known to be dominated by the first level RLT relaxation, see \cite{adams1990johnson}.
Adams and Johnson  \cite{adams1990johnson}  also show that the QAP bounds from  \cite{carraresi1992new,assad1985lower} are dominated by the first level QAP-RTL bound.
However, we show here more since our proof is not restrict to a particular skew-symmetric matrix.

Note that if $x\leq e$ are not redundant for $\bar{K}$, we have that $v_{LBB'}$ dominates $v_{RLT_1'}$. In particular we have the following result.
\begin{lemma} \label{notRedundant}
Suppose that  $x \leq e$ is not redundant for $\bar{K}$, then $v_{RLT_1'} \leq v_{LBB'} \leq v_{RLT_1}$.
\end{lemma}

\subsection{The LBB$^{*}$ bound}\label{sect:strongest}

 The linearization-based bound $v_{LBB^{*}}$ can be viewed as a strengthened   $v_{LBB'}$ bound.
 We show here by an example that $v_{LBB^{*}}$  may be stronger than $v_{RLT_1}$.

Let $Q_{1},\ldots,Q_{k}$ be linearizable matrices for a given BQP with linearization vectors $c_{1},\ldots,c_{k}$, respectively.
In general,  those matrices are not of the form given in  Lemma \ref{poly_gen1}.
Combining linearizable matrices $Q_{1},\ldots,Q_{k}$ together with the  symmetrized linearizable matrix from Lemma \ref{poly_gen1}, we obtain the following linear programming relaxation:
\begin{equation}\label{poly_lbbstar}
\begin{array}{ll}
& \underset{Y,\alpha,y}{\max} \left \{ b^{\mathrm T}y \;|\; B^{\mathrm T}Y +  Y^{\mathrm T}B+  \Diag(z) +  \sum_{i=1}^{k} \alpha_{i} Q_{i} \leq Q,\; B^{\mathrm T}y \leq 2Y^{\mathrm T}b  + z +C\alpha \right \},
\end{array}
\end{equation}
whose  dual is given by
\begin{equation}\label{poly_RLT_lbbstar}
\begin{array}{cl}
\underset{x \in \mathbb{R}^{m}, X \in \mathcal{S}^{m} }{\min} & \langle Q,X \rangle \\
\text{ s.t. }& (x,X) \in \mathcal{F} \\
& \langle Q_{i}, X \rangle = \langle c_{i} , x \rangle \quad \text{ for } i = 1,\ldots,k.
\end{array}
\end{equation}
where $\mathcal{F}$ is given in \eqref{RLT_fea}. Thus, the linear program \eqref{poly_RLT_lbbstar} is a strengthened relaxation
\eqref{poly_RLT} with additional constraints of the form $\langle Q_{i}, X \rangle = \langle c_{i} , x \rangle$.
Indeed, if $Q_{i}$ is linearizable with linearization vector $c_i$,
then $x^{\mathrm T}Q_{i}x = c_{i}^{\mathrm T}x$ for every $s$-$t$ path $x$. This yields a valid constraint $\langle Q_{i}, X \rangle = \langle c_{i}, x \rangle$.
In particular, when the set of matrices $\{ Q_{1},\ldots,Q_{k}\}$ span the set of linearizable matrices, the optimal  solution of relaxation \eqref{poly_RLT_lbbstar} is just $v_{LBB^{*}}$.

Recall the first iteration of GGL is simply the Gilmore-Lawler type bound, denoted by $v_{GL}$.
To the best of our knowledge, it was only known  that $v_{GL}\leq  v_{GGL}  \leq v_{RLT_1'}= v_{RLT_1}$ for the QAP.
Moreover, $v_{GGL}$ for the QAP was studied only for particular cases, not in general.
We summarize below  relations between several of the mentioned bounds for any BQP.

\begin{proposition} \label{prop:compare}
\begin{enumerate}
\item[(a)] Let  $x \leq e$ be redundant for $\bar{K}$, then
\[
v_{GL} \leq v_{GGL} \leq v_{LBB'} =v_{RLT_1'} = v_{RLT_1} \leq  v_{LBB^{*}}.
\]
\item[(b)]  Let  $x \leq e$ be not redundant for $\bar{K}$, then
\[
 v_{RLT_1'} \leq v_{LBB'}  \leq  v_{LBB^{*}}.
\]
\end{enumerate}
\end{proposition}

\begin{proof}
Part (a) follows from Theorem \ref{GLL_lbb}, Lemma \ref{LBBisRLT} and construction of $v_{LBB^{*}}$.
Part (b) follows from  Lemma \ref{notRedundant} and construction of $v_{LBB^{*}}$.
\end{proof}

It should be clear that for case (a) in the above proposition
we have $v_{RLT_1} = v_{LBB^{*}}$ whenever the spanning set of linearizable matrices
can be characterized by the linearizable matrices of the form $B^{\mathrm T}Y + Y^{\mathrm T}B  + \Diag(z)$.
Unfortunately we were not able to prove that  the  inequality $v_{RLT_{1}} \leq v_{LBB^{*}}$ is strict when  $x \leq e$ is redundant for $\bar{K}$.
On the other hand, we found examples when those two bounds are equal.
By using Proposition \ref{main_basis} we computed the spanning set of linearizable matrices for the QSPP
on tournament graphs, GRID1, GRID2 and PAR-K graphs up to certain size. We refer to \cite{hu2017solving}
for the definition of these acyclic graphs. We have also computed the spanning set of linearizable matrices
for the QAP for $n \leq 9$ by brute-force search. It turns out that the spanning sets of all mentioned instances
can be expressed as  $B^{\mathrm T}Y + Y^{\mathrm T}B  + \Diag(z)$.
This indicates that the set of (symmetrized) linearizable matrices obtained from Lemma \ref{poly_gen1} is considerably rich.

Note that for case (b) in Proposition \ref{prop:compare}, we do not compare $v_{LBB^{*}}$ and $v_{RLT_1}$,
as their relation is not clear in  general.
However, our example below shows that there are instances for which $v_{RLT_1} < v_{LBB^{*}}$.

\begin{example} \label{examp1}{\rm  
Consider the quadratic shortest path problem on the complete symmetric digraph $K^*_n$,  that is a digraph in which every pair of vertices is connected by a bidirectional edge.
Assume the incoming arcs to $s$ and outgoing arcs from $t$ are removed. We can find the spanning set of linearizable matrices as follows.
If $Q$ is linearizable with linearization vector $c$, then $\langle xx^{T},Q - \Diag(c) \rangle = 0$ for every $s$-$t$ path $x$. By enumerating all $s$-$t$ paths,
we have a linear system whose null space represents a set of linearizable matrices.
That set, together with the set of diagonal matrices $\{\Diag(e_{k}) \;|\; k = 1,\ldots,m \}$,
provides the spanning set of linearizable matrices for $K^*_n$. Here, $e_k$  denotes column $k$ of the identity matrix.

 For $n = 5$, there are $85$ linearly independent linearizable matrices in the spanning set.
  However there are only $59$ linearly independent linearizable matrices from the set
  $\{ B^{\mathrm T}(e_ie_j^{\mathrm T}) + (e_je_i^{\mathrm T})B + \Diag(e_{k}) \;|\; i = 1,\ldots,n \text{ and } j,k  = 1,\ldots,m \}$,
  see Lemma \ref{poly_gen1}. Here $B$ is the incidence matrix of the complete symmetric digraph with five vertices.

Now, we use these linearizable matrices as the cost matrix $Q$ in the QSPP. 
For each linearizable matrix, we have $v_{LBB^{*}} = 0$ as there exists an $\alpha$ such that $\mathcal{A}(\alpha) =  Q$ in \eqref{poly_lbb}.
However $v_{LBB'}$ is unbounded from below for some of the instances due to negative loops in $K^*_5$.
Table \ref{example} we shows numerical results for the  first level RLT bound $v_{RLT_1}$ for seven of the mentioned instances.
We also compute the strongest semidefinite programming relaxation $SDP_{NL}$ from \cite{hu2017solving} for the mentioned instances.
Optimal values are provided in the last column of the table.

\begin{table}[H]
	\small
	\centering
	\begin{tabular}{|c|cccc|} \hline
		instance & $v_{RLT_1}$ &  $v_{SDP_{NL}}$  &  $v_{LBB^{*}}$ & opt  \\ \hline
1  & -0.51757  &  -0.51757  &       0  &  0  \\
2  & -0.7399  &  -0.74882  &       0  &  0  \\
3  & -0.58818  &  -0.48862  &       0  &  0  \\
4  & -0.48526  &  -0.48526  &       0  &  0  \\
5  & -0.66547  &  -0.56165  &       0  &  0  \\
6  & -0.27015  &  -0.22945  &       0  &  0  \\
7  & -1.1405  &  -1.1405  &       0  &  0  \\  \hline
	\end{tabular}
	\caption{The bounds for QSPP instances on $K^*_n$}\label{example}
\end{table}

Table \ref{example} shows that $v_{LBB^{*}}$ dominates both $v_{RLT_1}$  and $v_{SDP_{NL}}$ for all instances. 
It is surprising that $v_{LBB^{*}}$ dominates also  the semidefinite programming bound.
}
\end{example}

\subsection{The extended linearization-based bounds} \label{sect:extendLin}
In Section \ref{sec:lbb} we  introduce the linearization-based bounding scheme by exploiting linearizable matrices.
In this section, we extend the notion of linearizable matrices, which enables us to construct a  linearization-type of a bound that turns to be equivalent to the first level RLT bound also for BQPs where $x\leq e$ is not redundant for $\bar{K}$.

For a given BQP, we call the cost matrix $Q$ extended linearizable if there exists a vector
$c$ such that $c^{\mathrm T}x \leq x^{\mathrm T}{Q}x$ for all $x \in K$. An example of extended linearizable matrix is given below.

\begin{lemma}
	Let $\Lambda \in {\mathcal S}_+^m$,  $\Omega \in \mathbb{R}^{m \times m}_{+}$. Then $\Lambda - \Omega$ is extended linearizable, and
	$$(2\Lambda e - \Omega e)^{\mathrm T}x - e^{\mathrm T}\Lambda e \leq x^{\mathrm T}(\Lambda - \Omega)x $$
	for any binary vector $x \in \{0,1\}^{m}$.
\end{lemma}
\begin{proof}
	It is not difficult to see  that $-x^{\mathrm T}\Omega e \leq -x^{\mathrm T}\Omega x $. It also holds that
	$$2e^{\mathrm T}\Lambda x - e^{\mathrm T}\Lambda e - x^{\mathrm T}\Lambda x = -(e-x)^{\mathrm T} \Lambda (e-x) \leq 0.$$
	The statement follows by summing both inequalities.
\end{proof}

Similar to the linearization-based bound, we can use the (extended) linearizable matrices to derive a lower bound for a given BQP.
For any fixed $Y\in \mathbb{R}^{n \times m},z \in \mathbb{R}^{m},\Lambda\in {\mathcal S}_+^m,\Omega \in \mathbb{R}^{m \times m}_{+}$ such that $B^{\mathrm T}Y +  Y^{\mathrm T}B +  \Diag(z)+ \Lambda - \Omega  \leq Q$,  the optimal value of the following problem
\begin{align*}
\min_{x\in \bar{K}} (2Y^{\mathrm T}b  + z + 2\Lambda e - \Omega e )^{\mathrm T}x,
\end{align*}
subtracted  by $e^{\mathrm T}\Lambda e $ is a lower bound for the BQP.
The dual of the above LP is
\begin{align*}
\underset{Y,z,y}{\max}   &\{ \;  b^{\mathrm T}y  \;|\;    B^{\mathrm T}y  \leq 2Y^{\mathrm T}b  + z + 2\Lambda e - \Omega e \}.
\end{align*}
In order to obtain the strongest bound of the above type, one has to solve the following maximization problem.
\begin{equation}\label{extend_dual}
\begin{array}{rrl}
v_{ExLBB}:=& \underset{Y,z,y,\Lambda= \Lambda^{\mathrm T}, \Omega}{\max}  &  b^{\mathrm T}y - \langle J, \Lambda \rangle\\
&  & B^{\mathrm T}Y +  Y^{\mathrm T}B +  \Diag(z)+ \Lambda - \Omega  \leq Q \\
& & B^{\mathrm T}y  \leq 2Y^{\mathrm T}b  + z + 2\Lambda e - \Omega e \\
&& \Lambda \geq 0,\Omega \geq 0.
\end{array}
\end{equation}

We call the solution of \eqref{extend_dual} the extended linearization-based bound. We conclude this section by showing that $v_{ExLBB}$ is equivalent to the first level RLT bound \eqref{RLT}.
\begin{theorem}
Suppose that  $x \leq e$ is not redundant for $\bar{K}$, then 	$v_{ExLBB} = v_{RLT_1}$.
\end{theorem}
\begin{proof} Proof follows by verifying that  \eqref{RLT} and  \eqref{extend_dual} are a primal-dual pair.
\end{proof}
The above result is interesting from a theoretical perspective. However, from a practical view the linearization-based bounds are more attractive due to the smaller number of variables and constraints.

\section{The QSPP linearization problem on DAGs} \label{section_qspp_linear}
In this section, we first introduce several assumptions and definitions.
Then, we derive necessary and sufficient conditions for an instance of the QSPP on a directed acyclic graph to be linearizable, see Theorem \ref{main_linear}.
We also show that those conditions can be verified in ${\mathcal O}(nm^{3})$ time.
This result is  a generalization  of the corresponding results for the QSPP on directed grid graphs from \cite{hu2018special}.

We have the following assumptions in this section:
\begin{enumerate} [topsep=0pt,itemsep=-1ex,partopsep=1ex,parsep=2ex,label=(\roman*)]
	\item $G$ is a directed acyclic graph;
	\item the vertices  $v_{1},\ldots,v_{n}$ are topologically sorted, that is, $(v_{i},v_{j}) \in A $ implies $i < j$;
	\item  for each vertex $v$, there exists at least one $s$-$t$ path containing $v$;
	\item the diagonal entries of the cost matrix $Q$ are zeros.
\end{enumerate}

It is  well-known that a topological ordering of a directed acyclic graph can be computed efficiently.
	We also note that assumptions (iii) and (iv) do not restrict the generality.
 For instance, assumption (iv) is not restrictive as $Q$ is linearizable if and only if $Q+\Diag(c)$
is linearizable for any cost vector $c$, see Lemma 4.1 in \cite{hu2018special}.

Here, we choose and fix an arbitrary labeling of the arcs.
The vertices $v_{2},\ldots,v_{n-1}$, e.g., those between the source vertex $v_{1}$ and the target vertex $v_{n}$, are called the \emph{transshipment vertices}.
For each transshipment vertex, we pick the outgoing arc with the smallest index and call it a \emph{non-basic arc}.
The remaining arcs are \emph{basic}. Thus there are $n-2$ non-basic arcs and $m-n+2$ basic arcs.

Note that for a  linearizable cost matrix its linearization vector is not necessarily unique.
However, we would like to restrict our analysis to the  linearization vectors that are in a unique, reduced form.
For this purpose we introduce the following definitions, see also  \cite{hu2018special}.

\begin{definition}\label{equivalence}
	We say that the cost vectors $c_{1}$ and $c_{2}$ are {equivalent} if  $c_{1}^{\mathrm T}x = c_{2}^{\mathrm T}x$ for all $x \in \mathcal{P}$.
\end{definition}

\begin{definition}\label{Def_reducedform}
	The \emph{reduced form} of a cost vector $c$ is an equivalent cost vector $R(c)$ such that $(R(c))_{e} = 0$ for every non-basic arc $e$.
\end{definition}

The existence of the reduced form of the cost vector $c$ follows from the following transformation.
Let $v$ be a transshipment vertex, and $f$ the non-basic arc going  from $v$. Define $\hat{c}$ as follows:
\begin{align} \label{reduced_form}
\hat{c}_{e} := \begin{cases}
c_{e} - c_{f} & \text{ if } e \text{ is an outgoing arc from vertex } v,\\
c_{e} + c_{f}  & \text{ if } e \text{ is an incoming arc to } v,\\
c_{e} & \text{ otherwise}.
\end{cases}
\end{align}
It is not difficult to verify that $\hat{c}$ and $c$ are equivalent.
Furthermore, if we apply this transformation at each transshipment vertex $v$ in the reverse topological order i.e., $v_{n-1},\ldots,v_{2}$,
then the obtained cost vector, after $n-2$ transformations, is in the reduced form.
Moreover the resulting cost vector is equivalent to $c$.
Let us define now critical paths, see also  \cite{hu2018special}.

\begin{definition}\label{DEF:critical_path}
	For a basic arc $e = (u,v)$, the associated   \emph{critical path} $P_{e}$ is an $s$-$t$ path containing arc $e$ and determined as follows.
	Choose an arbitrary $s$-$u$ path $P_{1}$, and take for $P_{2}$ the unique $v$-$t$ path with only non-basic arcs.
	Then, the critical path $P_{e} = (P_{1},P_{2})$
	is the concatenation of the paths $P_{1}$ and $P_{2}$.
\end{definition}
The uniqueness of $P_{2}$ in Definition \ref{DEF:critical_path} follows from the fact that each transshipment vertex has exactly one
outgoing arc that is non-basic and $G$ is acyclic. Clearly,  to each basic arc $e$  we can associate one critical path $P_{e}$ as given above.

The following result shows that for a linearizable cost matrix $Q$   there exists a unique linearization vector in reduced form.
The uniqueness is up to the choice of non-basic arcs and critical paths.

\begin{proposition}\label{lc_uniqueness}
	Let $Q\in \R^{m\times m}$ be a linearizable cost matrix for the QSPP, and $c \in \R^{m}$ its linearization vector.
	Then, the reduced form of $c$, $R(c) \in \R^{m}$, is uniquely determined by the costs of the critical paths in the underlying graph $G$.
\end{proposition}
\begin{proof}
	Let $M$ be a binary matrix whose rows correspond to the $s$-$t$ paths in $G$ and columns correspond to the basic arcs.
	In particular, $M_{P,e} = 1$ if and only if the path $P$ contains the basic arc $e$.
	Let $b$ be the vector whose elements contain quadratic costs of  $s$-$t$  paths.
	Let $\hat{c}_{B} \in \R^{m-n+2}$ is the subvector of $R(c)$ composed of the elements corresponding to the basic arcs.
	Then, $\hat{c}_{B}$ satisfies the linear system $M\hat{c}_{B} = b$.
	In order to show the uniqueness of $\hat{c}_{B}$ it suffices to prove that the rank of $M$ equals $m-n+2$, which is the number of the basic arcs.
	
	Let $\bar{M}$ be a square submatrix of $M$ of size $(m-n+2) \times (m-n+2)$ whose rows correspond to the critical paths.
	Let $\mathcal{C}_{i}$ be the set of the basic arcs emanating from vertex $v_{i}$ for $i = 1,\ldots,n-1$.
	Since the sets $\mathcal{C}_{1}, \ldots, \mathcal{C}_{n-1}$
	partition the set of the basic arcs, they can be used to index the matrix $\bar{M}$.
	Upon rearrangement, $\bar{M}$ is a block matrix such that the $(i,j)$th block $\bar{M}^{ij}$ is the submatrix  whose rows and columns correspond to $\mathcal{C}_{i}$
	and $\mathcal{C}_{j}$, respectively.  It is readily seen that every diagonal block $\bar{M}^{ii}$ is an identity-matrix. Furthermore, the block $\bar{M}^{ij}$
	is a zero matrix for $i < j$.
	To see this, we first recall that the vertices are topologically ordered. Then, note that for the critical path that is associated to the arc $e = (i,j)$,
	all  arcs visited after $e$ are non-basic  by construction. Thus, the rank of $\bar{M}$ is $m-n+2$, and this finishes the proof.	
\end{proof}	

From the previous proposition, it follows  that for a linearizable cost matrix $Q$
its linearization vector can be computed easily from the costs of the critical paths.
However, the above calculation of the unique linear cost vector in reduced form can be performed even when the linearizability of $Q$ is not known.
Since the resulting vector  does not have to be a linearization vector,  we call it  pseudo-linearization vector.
In particular, we have the following definition, see also  \cite{hu2018special}.

\begin{definition}
	The pseudo-linearization vector
	of the cost matrix $Q \in \R^{m \times m}$ is the unique cost vector $p \in \R^{m}$ in reduced form such that $x^{\mathrm T}Qx = p^{\mathrm T}x$
 for every critical path $x$.
\end{definition}

Here, the uniqueness is up to the choice of non-basic arcs and critical paths.
Recall that the pseudo-linearization vector can be computed also for a non-linearizable cost matrix.
The following lemma shows that  for linearizable $Q$ its pseudo-linearization vector  coincides with the linearization vector in reduced form.

\begin{lemma}\label{lc_pseudo}
	Let $Q$ be linearizable.  Then the corresponding linearization vector in reduced form and the pseudo-linearization vector   are equal.
\end{lemma}
\begin{proof}
	Let $c$ be the linearization vector of $Q$ in reduced form, and $p$ the pseudo-linearization vector of $Q$.
	Then, $c^{\mathrm T}x = p^{\mathrm T}x$ for each critical path $x$.  From Proposition \ref{lc_uniqueness} it follows  that $c = p$ since both cost vectors are in the reduced form.
\end{proof}

If we change the input target vertex from $t$ to another vertex $v$, some arcs and vertices have to be removed from $G$ in order to satisfy assumption (iii).
This results in a reduced QSPP instance. To simplify the presentation, we introduce the following notation.

\begin{table}[H]
	\centering
	{\footnotesize
		\begin{tabular}{|>{\centering\arraybackslash}m{2.9cm}|p{12cm}|}
			\hline
			Notation & \multicolumn{1}{|>{\centering\arraybackslash}m{110mm}|}{Definition}\\ \hline
			$G_{v} = (V_{v},A_{v})$ & an induced subgraph of $G$ for which assumption (iii) is satisfied with target vertex $v$ \\[.1cm]
			$Q_{v}$ & an ${|A_{v}|\times |A_{v}|}$ submatrix of $Q$ whose rows and columns correspond to the arcs in $A_{v}$ \\[.1cm]
			$R_{v}(\cdot)$ & the linear {function} that maps a cost vector on $G_{v}$ to its reduced form \\[.1cm]
			$p_{v}$ & the pseudo-linearization vector of $Q_{v}$ \\
			\hline
		\end{tabular} \caption{Notation with respect to target vertex $v$.} \label{target_notattion}
	}
\end{table}

The next result from \cite{hu2018special} establishes a relationship between the linearization vector of $Q_{t}$ and the linearization vector of $Q_{v}$ where $(v,t)\in A$.
\begin{lemma}\emph{\cite{hu2018special}} \label{lc_branch}
	The cost vector $c \in \R^{m}$ is a linearization of $Q_{t} \in \R^{m \times m}$ if and only if
	the cost vector $T_{e}(c) \in \R^{|A_{v}|}$   given by
	\begin{equation}
	\big(T_{e}(c)\big)_{e'} = \begin{cases}
	c_{e'} - 2 \cdot q_{e,e'}  & \text{ if } e' = (u,w) \in A_{v} \text{ and } u \neq s \\
	c_{e'}  - 2 \cdot q_{e,e'} + c_{e}  & \text{ if } e' = (u,w) \in A_{v} \text{ and } u = s \\
	\end{cases},
	\end{equation}
	is a linearization  of $Q_{v}$ for every vertex $v$ such that $e = (v,t)\in A$.
\end{lemma}

Now, we are ready to prove the main result in this section.

\begin{theorem}\label{main_linear}
	Let $R_{v}(\cdot)$ and $p_{v}$ be defined as in Table \ref{target_notattion}, and  $T_{e}$ given as in Lemma \ref{lc_branch}.
	Then,  the pseudo-linearization vector $p_{t}$ is a linearization  of $Q_{t}$ if and only if $(R_{u} \circ T_{e})(p_{v}) = p_{u}$ for every $e = (u,v) \in A$.
\end{theorem}
\begin{proof}
	If $Q_{t}$ is linearizable, then it follows from Lemma \ref{lc_pseudo} and Lemma \ref{lc_branch} that for every vertex $v$, the cost matrix $Q_{v}$ is linearizable with the linearization vector $p_{v}$.
	Thus, $(R_{u} \circ T_{e})(p_{v}) = p_{u}$ for every arc $e = (u,v) \in A$.
	
	Conversely, assume that $Q_{t}$ is not linearizable. Then, from Lemma \ref{lc_branch}, it follows that there exists an arc $e = (v,t) \in A$ such that
	the vector $T_{e}(p_{t})$ is not a linearization vector of $Q_{v}$. Let us distinguish the following two cases:
	
	\begin{enumerate} [topsep=0pt,itemsep=-1ex,partopsep=1ex,parsep=2ex,label=(\roman*)]
		\item If $Q_{v}$ is linearizable, then $p_{v}$ is its linearization vector and $(R_{v} \circ T_{e})(p_{t}) \neq p_{v}$ by Lemma \ref{lc_pseudo}.
		\item If $Q_{v}$ is not linearizable, then we again distinguish two cases. Thus, we repeat the whole argument to $Q_{v}$.
	\end{enumerate}
	
	\noindent This recursive process must eventually end up with case (i), as the number of vertices in the underlying graph decreases in each recursion step, and every cost matrix on a graph with at most three vertices is linearizable. Thus, we obtain $(R_{u} \circ T_{e})(p_{v}) \neq p_{u}$ for some $e = (u,v) \in A$.
\end{proof}

Note that the iterative procedure from Theorem  \ref{main_linear} provides an answer to the QSPP linearization problem.
Moreover, it returns the linearization vector in reduced form if such exists.  We present the pesudo-algorithm for the QSPP linearization problem below.

\begin{algorithm}[H]
	\caption{The QSPP linearization algorithm on DAGs}\label{main_lbb_alg}
	\begin{algorithmic}[1]
		\State \textbf{Input:} A QSPP instance with $G=(V,A)$ acyclic, $s,t\in V$ and cost matrix $Q$.
		\State \textbf{Output:} A linearization vector $c$ of $Q$, if it exists
		\Procedure{isLinearizable}{$G,Q$}
		\For{$v \in V$}
		\State compute the pseudo-linearization $p_{v}$, see Prop \ref{lc_uniqueness}
		\EndFor
		\State \textbf{end for}
		\For{$e = (u,v) \in A$}
		\If{$(R_{u} \circ T_{e})(p_{v}) \neq p_{u}$}
		\State $Q$ is not linearizable
		\State return
		\EndIf
		\State \textbf{end if}
		\EndFor
		\State \textbf{end for}
		\State $c \leftarrow p_t$ is the linearization vector of $Q$\;
		\EndProcedure
	\end{algorithmic}
\end{algorithm}


The quadratic cost of an $s$-$t$ path can be computed in ${\mathcal O}(m^2)$ steps, and thus we need ${\mathcal O}(m^3)$ steps for the $m-n+2$ critical paths.
The pseudo-linearization of $Q$ can be obtained in ${\mathcal O}(m^2)$ steps by solving a linear system whose left-hand-side
is a lower triangular square matrix $\bar{M}$ of order $m-n+2$, see Proposition \ref{lc_uniqueness}.
Since there are $n$ vertices, computing all  pseudo-linearizations requires ${\mathcal O}(nm^3)$ steps.
The rest of the computation takes at most ${\mathcal O}(m^3)$ steps.
Thus the complexity of the algorithm  given in Theorem \ref{main_linear} is ${\mathcal O}(nm^3)$.

Now we show that the linearization algorithm also characterize the set of linearizable matrices.
	\begin{proposition}\label{main_basis}
	Let $G$ be an acyclic digraph. The matrices $Q_{1},\ldots,Q_{k}$ spanning the set of linearizable matrices can be computed in polynomial time.
	\end{proposition}
	\begin{proof}
		Let $v_{i}$ be a transshipment vertex, and $f$ the non-basic arc going from $v_{i}$. Let $M_{i}$ be the matrix with ones on the diagonals, and the following non-zero off-digonal entries,
		\begin{align} \label{reduced_form_mat}
		(M_{i})_{e,f} := \begin{cases}
		-1 & \text{ if } e \text{ is an outgoing arc from vertex } v_{i},\\
		1  & \text{ if } e \text{ is an incoming arc to } v_{i}.\\
		\end{cases}
		\end{align}
		It it easy to see that for a given cost vector $c$,
$\hat{c} = M_{i}c$ gives an alternative representation for \eqref{reduced_form}. Thus, the reduced form $R_{u}(c)$ of $c$ is a linear transformation given by $R_{u}(c)=(M_{n-1}M_{n-2}\cdots M_{2})c$. Similarly, the linear operator $T_{e}$ from  Lemma \ref{lc_branch}, and the pseudo-linearization vector $p_{v}$  can be obtained from linear transformations. Therefore, we can find a matrix $L$ in polynomial time such that $L\text{vec}(Q) = 0$ if and only if $Q$ is linearizable. If $\text{vec}(Q_{1}),\ldots,\text{vec}(Q_{k})$ span the null space of $L$, then $Q$ is linearizable if and only if $\text{vec}(Q) = \text{vec}(\sum_{i=1}^{k}\alpha_{i}Q_{i})$ for some $\alpha \in \mathbb{R}^{k}$.
	\end{proof}
The previous proposition can be used to compute the spanning set of linearizable matrices for the QSPP on DAGs.

\section{Conclusions}\label{sec_conclusions}

In this paper, we present several  applications of the linearization problem for binary quadratic problems.
In particular,  we propose a new lower bounding scheme which follows from a simple certificate for a quadratic function to be non-negative.
Each linearization-based relaxation depends on the chosen  set of linearizable matrices.
This  allows us to compute a number of different lower bounds.
One can obtain the best possible linearization-based bound in the case that the complete characterization of the  set of linearizable matrices is known.

In Theorem \ref{GLL_lbb}, we prove that the Generalized Gilmore-Lawler bound obtained from Algorithm \ref{GGL}, and for any choice of a skew-symmetric matrix,
is bounded by  $v_{LBB'}$, see \eqref{poly_lbb1}.
We also show that $v_{LBB'}$ coincides with   the first level RLT bound by Adams and Sherali \cite{adams1990linearization}
when the upper bounds  on variables are implied by the rest of the constraints, see Lemma \ref{LBBisRLT}.
 This also implies that all Generalized Gilmore-Lawler bounds  are bounded by the first level RLT bound for the mentioned setting.
Similar result was already observed  in the context of the quadratic assignment problem, but it was not known for BQPs in general.
For BQPs where upper bounds on variables are not implied by constraints, Lemma \ref{notRedundant} establishes the relation between  $v_{LBB'}$ and  $v_{RLT_1}$.
In Proposition \ref{prop:compare}, we relate all here presented bounds with the strongest linearization-based bound $v_{LBB^{*}}$.
Our Example \ref{examp1} demonstrates the strength of that bound.

Finally, we provide a polynomial-time algorithm to solve the linearization problem of the quadratic shortest path problem on directed acyclic graphs, see Algorithm  \ref{main_lbb_alg}.
 Our algorithm  also yields the complete characterization of the  set of linearizable matrices for the QSPP on DAGs.
  Thus, we are  able to compute  the strongest linearization-based bound $v_{LBB^*}$ for the QSPP on DAGs.
Our numerical experiments show that  $v_{LBB^*}$ and  $v_{LBB'}$ coincides for all tested instances.

For future research,    it would be interesting  to further investigate strength of linearization-based bounds for types of linearizable matrices that do not fall into the
case of Lemma \ref{poly_gen1}.
   Finally, let us note that the results from this paper (partially) address  questions posed by Çela et al.~\cite{CELA2018818}
   related to computing good bounds by exploiting polynomially solvable special cases.

\bibliographystyle{plainnat}
\bibliography{sample}

\begin{thebibliography}{38}
\providecommand{\natexlab}[1]{#1}
\providecommand{\url}[1]{\texttt{#1}}
\expandafter\ifx\csname urlstyle\endcsname\relax
  \providecommand{\doi}[1]{doi: #1}\else
  \providecommand{\doi}{doi: \begingroup \urlstyle{rm}\Url}\fi

\bibitem[Adams and Johnson(1994)]{adams1990johnson}
Warren~P. Adams and Terri~A. Johnson.
\newblock Improved linear programming-based lower bounds for the quadratic
  assignment problem.
\newblock In \emph{Proceedings of the DIMACS Workshop on Quadratic Assignment
  Problems, DIMACS Series in Discrete Mathematics and Theoretical Computes
  Sciences}, pages 43--–75. AMS, 1994.

\bibitem[Adams and Sherali(1986)]{adams1986tight}
Warren~P. Adams and Hanif~D. Sherali.
\newblock A tight linearization and an algorithm for zero-one quadratic
  programming problems.
\newblock \emph{Management Science}, 32\penalty0 (10):\penalty0 1274--1290,
  1986.

\bibitem[Adams and Sherali(1990)]{adams1990linearization}
Warren~P. Adams and Hanif~D. Sherali.
\newblock Linearization strategies for a class of zero-one mixed integer
  programming problems.
\newblock \emph{Operations Research}, 38\penalty0 (2):\penalty0 217--226, 1990.

\bibitem[Adams and Waddell(2014)]{adams2014linear}
Warren~P. Adams and Lucas Waddell.
\newblock Linear programming insights into solvable cases of the quadratic
  assignment problem.
\newblock \emph{Discrete Optimization}, 14:\penalty0 46--60, 2014.

\bibitem[Assad and Xu(1985)]{assad1985lower}
Arjang~A. Assad and Weixuan Xu.
\newblock On lower bounds for a class of quadratic 0,1 programs.
\newblock \emph{Operations Research Letters}, 4\penalty0 (4):\penalty0
  175--180, 1985.

\bibitem[Billionnet et~al.(2009)Billionnet, Elloumi, and
  Plateau]{billionnet2009improving}
Alain Billionnet, Sourour Elloumi, and Marie-Christine Plateau.
\newblock Improving the performance of standard solvers for quadratic 0-1
  programs by a tight convex reformulation: The {QCR} method.
\newblock \emph{Discrete Applied Mathematics}, 157\penalty0 (6):\penalty0
  1185--1197, 2009.

\bibitem[Buchheim and Traversi(2018)]{buchheim2015quadratic}
Christoph Buchheim and Emiliano Traversi.
\newblock Quadratic 0--1 optimization using separable underestimators.
\newblock \emph{INFORMS Journal on Computing}, 30\penalty0 (3):\penalty0
  421--624, 2018.

\bibitem[Burkard et~al.(2009)Burkard, Dell'Amico, and Martello]{Burkard:2009}
Rainer Burkard, Mauro Dell'Amico, and Silvano Martello.
\newblock \emph{Assignment Problems}.
\newblock Society for Industrial and Applied Mathematics, Philadelphia, PA,
  USA, 2009.
\newblock ISBN 0898716632, 9780898716634.

\bibitem[Burkard et~al.(2012)Burkard, Dell'Amico, and
  Martello]{burkard2012assignment}
Rainer Burkard, Mauro Dell'Amico, and Silvano Martello.
\newblock \emph{Assignment problems, revised reprint}, volume 106.
\newblock Siam, 2012.

\bibitem[Carraresi and Malucelli(1992)]{carraresi1992new}
Paolo Carraresi and Federico Malucelli.
\newblock A new lower bound for the quadratic assignment problem.
\newblock \emph{Operations Research}, 40\penalty0 (1):\penalty0 S22--S27, 1992.

\bibitem[Cela et~al.(2016)Cela, Deineko, and Woeginger]{cela2016linearizable}
Eranda Cela, Vladimir~G. Deineko, and Gerhard~J. Woeginger.
\newblock Linearizable special cases of the {QAP}.
\newblock \emph{Journal of Combinatorial optimization}, 31\penalty0
  (3):\penalty0 1269--1279, 2016.

\bibitem[{\'C}usti{\'c} and Punnen(2018)]{custic2018characterization}
Ante {\'C}usti{\'c} and Abraham~P. Punnen.
\newblock A characterization of linearizable instances of the quadratic minimum
  spanning tree problem.
\newblock \emph{Journal of Combinatorial Optimization}, 35\penalty0
  (2):\penalty0 436--453, 2018.

\bibitem[{\'C}usti{\'c} et~al.(2017){\'C}usti{\'c}, Sokol, Punnen, and
  Bhattacharya]{custic2017characterization}
Ante {\'C}usti{\'c}, Vladyslav Sokol, Abraham~P. Punnen, and Binay
  Bhattacharya.
\newblock The bilinear assignment problem: complexity and polynomially solvable
  special cases.
\newblock \emph{Mathematical Programming}, 166\penalty0 (1-2):\penalty0
  185--205, 2017.

\bibitem[Çela et~al.(2018)Çela, Deineko, and Woeginger]{CELA2018818}
Eranda Çela, Vladimir Deineko, and Gerhard~J. Woeginger.
\newblock New special cases of the quadratic assignment problem with diagonally
  structured coefficient matrices.
\newblock \emph{European Journal of Operational Research}, 267\penalty0
  (3):\penalty0 818 -- 834, 2018.

\bibitem[Erdo{\u{g}}an and Tansel(2011)]{erdougan2011two}
G{\"u}ne{\c{s}} Erdo{\u{g}}an and Barbaros~{\c{C}}. Tansel.
\newblock Two classes of quadratic assignment problems that are solvable as
  linear assignment problems.
\newblock \emph{Discrete Optimization}, 8\penalty0 (3):\penalty0 446--451,
  2011.

\bibitem[Frank and Sotirov()]{deMeijer2019}
de~Meijer Frank and Renata Sotirov.
\newblock The quadratic cycle cover problem.
\newblock \emph{Journal of Combinatorial Optimization, accepted, 2020}.

\bibitem[Frieze and Yadegar(1983)]{frieze1983quadratic}
Alan~M. Frieze and Joseph Yadegar.
\newblock On the quadratic assignment problem.
\newblock \emph{Discrete Applied Mathematics}, 5\penalty0 (1):\penalty0 89--98,
  1983.

\bibitem[Gilmore(1962)]{gilmore1962optimal}
Paul~C. Gilmore.
\newblock Optimal and suboptimal algorithms for the quadratic assignment
  problem.
\newblock \emph{Journal of the society for industrial and applied mathematics},
  10\penalty0 (2):\penalty0 305--313, 1962.

\bibitem[Hahn and Grant(1998)]{hahn1998lower}
Peter Hahn and Thomas Grant.
\newblock Lower bounds for the quadratic assignment problem based upon a dual
  formulation.
\newblock \emph{Operations Research}, 46\penalty0 (6):\penalty0 912--922, 1998.

\bibitem[Hahn et~al.(1998)Hahn, Grant, and Hall]{hahn1998branch}
Peter Hahn, Thomas Grant, and Nat Hall.
\newblock A branch-and-bound algorithm for the quadratic assignment problem
  based on the hungarian method.
\newblock \emph{European Journal of Operational Research}, 108\penalty0
  (3):\penalty0 629--640, 1998.

\bibitem[Hu and Sotirov(2018)]{hu2018special}
Hao Hu and Renata Sotirov.
\newblock Special cases of the quadratic shortest path problem.
\newblock \emph{Journal of Combinatorial Optimization}, 35\penalty0
  (3):\penalty0 754--777, 2018.

\bibitem[Hu and Sotirov(2019)]{hu2017solving}
Hao Hu and Renata Sotirov.
\newblock On solving the quadratic shortest path problem.
\newblock \emph{INFORMS Journal on Computing}, 2019.
\newblock URL \url{https://pubsonline.informs.org/doi/10.1287/ijoc.2018.0861}.

\bibitem[Kabadi and Punnen(2011)]{kabadi2011n}
Santosh~N. Kabadi and Abraham~P. Punnen.
\newblock An {O}($n^4$) algorithm for the {QAP} linearization problem.
\newblock \emph{Mathematics of Operations Research}, 36\penalty0 (4):\penalty0
  754--761, 2011.

\bibitem[Koopmans and Beckmann(1957)]{koopmans1957assignment}
Tjalling~C. Koopmans and Martin Beckmann.
\newblock Assignment problems and the location of economic activities.
\newblock \emph{Econometrica: journal of the Econometric Society}, pages
  53--76, 1957.

\bibitem[Lasserre(2001)]{lasserre2001global}
Jean~B. Lasserre.
\newblock Global optimization with polynomials and the problem of moments.
\newblock \emph{SIAM Journal on optimization}, 11\penalty0 (3):\penalty0
  796--817, 2001.

\bibitem[Laurent(2009)]{laurent2009sums}
Monique Laurent.
\newblock Sums of squares, moment matrices and optimization over polynomials.
\newblock In \emph{Emerging applications of algebraic geometry}, pages
  157--270. Springer, 2009.

\bibitem[Lawler(1963)]{lawler1963quadratic}
Eugene~L. Lawler.
\newblock The quadratic assignment problem.
\newblock \emph{Management Science}, 9\penalty0 (4):\penalty0 586--599, 1963.

\bibitem[Lendl et~al.(2019)Lendl, {\'C}usti{\'c}, and
  Punnen]{lendl2017combinatorial}
Stefan Lendl, Ante {\'C}usti{\'c}, and Abraham~P. Punnen.
\newblock Combinatorial optimization problems with interaction costs:
  Complexity and solvable cases.
\newblock \emph{Discrete Optimization}, 33:\penalty0 101--117, 2019.

\bibitem[Murakami and Kim(1997)]{murakami1997comparative}
Kazutaka Murakami and Hyong~S Kim.
\newblock Comparative study on restoration schemes of survivable atm networks.
\newblock In \emph{INFOCOM'97. Sixteenth Annual Joint Conference of the IEEE
  Computer and Communications Societies. Driving the Information Revolution.,
  Proceedings IEEE}, volume~1, pages 345--352. IEEE, 1997.

\bibitem[Nie and Schweighofer(2007)]{nie2007complexity}
Jiawang Nie and Markus Schweighofer.
\newblock On the complexity of putinar's positivstellensatz.
\newblock \emph{Journal of Complexity}, 23\penalty0 (1):\penalty0 135--150,
  2007.

\bibitem[Punnen(2001)]{punnen2001combinatorial}
Abraham~P. Punnen.
\newblock Combinatorial optimization with multiplicative objective function.
\newblock \emph{International Journal of Operations and Quantitative
  Management}, 7\penalty0 (3):\penalty0 205--210, 2001.

\bibitem[Punnen and Kabadi(2013)]{punnen2013linear}
Abraham~P. Punnen and Santosh~N. Kabadi.
\newblock A linear time algorithm for the {K}oopmans--{B}eckmann {QAP}
  linearization and related problems.
\newblock \emph{Discrete Optimization}, 10\penalty0 (3):\penalty0 200--209,
  2013.

\bibitem[Punnen et~al.(2017)Punnen, Woods, and Kabadi]{punnenWoods}
Abraham~P. Punnen, Brad~D. Woods, and Santosh~N. Kabadi.
\newblock A characterization of linearizable instances of the quadratic
  traveling salesman problem.
\newblock 2017.
\newblock URL \url{http://arxiv.org/abs/1708.07217}.

\bibitem[Punnen et~al.(2019)Punnen, Pandey, and
  Friesen]{punnen2018representations}
Abraham~P. Punnen, Pooja Pandey, and Michael Friesen.
\newblock Representations of quadratic combinatorial optimization problems: A
  case study using the quadratic set covering problem.
\newblock \emph{Computers and Operations Research}, 112,104769, 2019.

\bibitem[Rostami and Malucelli(2015)]{Rostami:QMST}
Borzou Rostami and Federico Malucelli.
\newblock Lower bounds for the quadratic minimum spanning tree problem based on
  reduced cost computation.
\newblock \emph{Computers and Operations Research}, 64:\penalty0 178--188,
  2015.

\bibitem[Rostami et~al.(2018)Rostami, Chassein, Hopf, Frey, Buchheim,
  Malucelli, and Goerigk]{rostami2018quadratic}
Borzou Rostami, Andr{\'e} Chassein, Michael Hopf, Davide Frey, Christoph
  Buchheim, Federico Malucelli, and Marc Goerigk.
\newblock The quadratic shortest path problem: complexity, approximability, and
  solution methods.
\newblock \emph{European Journal of Operational Research}, 268\penalty0
  (2):\penalty0 473--485, 2018.

\bibitem[Sen et~al.(2001)Sen, Pillai, Joshi, and Rathi]{Sen:01}
Suvrajeet Sen, Rekha Pillai, Shirish Joshi, and Ajay~K. Rathi.
\newblock A mean-variance model for route guidance in advanced traveler
  information systems.
\newblock \emph{Transportation Science}, 35\penalty0 (1):\penalty0 37--49,
  2001.

\bibitem[Sivakumar and Batta(1994)]{sivakumar1994variance}
Raj~A. Sivakumar and Rajan Batta.
\newblock The variance-constrained shortest path problem.
\newblock \emph{Transportation Science}, 28\penalty0 (4):\penalty0 309--316,
  1994.

\end{thebibliography}
\end{document}